\documentclass[10pt]{amsart}
\input{diagrams}

\newtheorem{proposition}{Proposition}[section]
\newtheorem{lemma}[proposition]{Lemma}
\newtheorem{corollary}[proposition]{Corollary}
\newtheorem{theorem}[proposition]{Theorem}

\theoremstyle{definition}

\newtheorem{example}[proposition]{Example}

\theoremstyle{remark}

\newcommand{\thlabel}[1]{\label{th:#1}}
\newcommand{\thref}[1]{Theorem~\ref{th:#1}}
\newcommand{\selabel}[1]{\label{se:#1}}
\newcommand{\seref}[1]{Section~\ref{se:#1}}
\newcommand{\lelabel}[1]{\label{le:#1}}
\newcommand{\leref}[1]{Lemma~\ref{le:#1}}
\newcommand{\prlabel}[1]{\label{pr:#1}}
\newcommand{\prref}[1]{Proposition~\ref{pr:#1}}
\newcommand{\colabel}[1]{\label{co:#1}}
\newcommand{\coref}[1]{Corollary~\ref{co:#1}}

\newcommand{\exlabel}[1]{\label{ex:#1}}

\newcommand{\eqlabel}[1]{\label{eq:#1}}
\newcommand{\equref}[1]{(\ref{eq:#1})}

\newcommand{\Hom}{{\rm Hom}}
\newcommand{\Rat}{{\rm Rat}}
\newcommand{\End}{{\rm End}}

\newcommand{\Ker}{{\rm Ker}\,}

\newcommand{\im}{{\rm Im}\,}

\def\ot{\otimes}

\def\ZZ{{\mathbb Z}}

\newcommand{\Cc}{\mathcal{C}}

\newcommand{\Mm}{\mathcal{M}}

\newcommand{\Oo}{\mathcal{O}}

\newcommand{\Tt}{\mathcal{T}}
\def\*C{{}^*\hspace*{-1pt}{\Cc}}
\def\text#1{{\rm {\rm #1}}}

\def\ol{\overline}

\begin{document}
\title[Comodules over semiperfect corings]{Comodules over semiperfect corings}
\author{S. Caenepeel}
\address{Faculty of Applied Sciences,
Vrije Universiteit Brussel, VUB, B-1050 Brussels, Belgium}
\email{scaenepe@vub.ac.be}
\urladdr{http://homepages.vub.ac.be/\~{}scaenepe/}
\author{M. Iovanov}
\address{Faculty of Mathematics, University of Bucharest, 
Str. Academiei 14, RO-70109 Bucharest, Romania}
\email{myo30@lycos.com}
\thanks{Research supported by the bilateral project
``Hopf Algebras in Algebra, Topology, Geometry and Physics" of the Flemish and
Romanian governments.}


\begin{abstract}
We discuss when the Rat functor associated to a coring satisfying the left
$\alpha$-condition is exact.
We study the category of
comodules over a semiperfect coring. We characterize semiperfect corings over
artinian rings and over qF-rings. 
\end{abstract}
\maketitle

\section*{Introduction}
The aim of this note is to generalize properties of semiperfect coalgebras
over fields,
as discussed in \cite{NGT}, see also \cite{DascalescuNR}, to semiperfect
corings. We also extend some results given in \cite{BW}. 
\\
Corings were introduced by Sweedler \cite{Sweedler65}. A coring over a
(possibly noncommutative) ring $R$ is a
coalgebra (or comonoid) in the category of $R$-bimodules.
Since the beginning of the 21st century, there has been a renewed interest 
in corings and comodules over a coring, iniated by Brzezi\'nski's paper \cite{B}.
The key point is that Hopf modules and most of their generalizations
(relative Hopf modules, graded modules, Yetter-Drinfeld modules and many more)
are comodules over a certain coring. This observation appeared
in MR 2000c 16047 written by Masuoka, who tributed it to Takeuchi, but apparently
it was already known by Sweedler, at least in the case of Hopf modules.
It has lead to a unified and simplified treatment of the above mentioned modules,
and new viewpoints on subjects like descent theory and Galois theory. For an
extensive treatment, we refer to \cite{BW}.\\
In this paper, we study semiperfect corings. A coring is called right semiperfect
if it satisfies the left $\alpha$-condition, and the (abelian) category of right $\Cc$-comodules
is semiperfect, which means that every simple object has a projective cover.
It turns out that this notion is closely related to rationality properties of modules over the
dual of the coring (which is a ring). Rationality properties have been studied
in \cite{A} and \cite{CVW}. The Rat functor sends a module over the dual of the coring
to its largest rational submodule. It can be described using the category $\sigma[M]$.
The category $\sigma[M]$ is discussed briefly in \seref{1}, and the Rat functor is
introduced in \seref{2}. General facts on the category $\sigma[M]$ show that the exactness of the Rat functor is connected to some topological properties of the base ring $R$, more precisely the $M$-adic topology on $M$. In the case of corings, the $\Cc$-adic topology on ${}^{*}\Cc$ coincides with the finite topology, motivating a general study of the properties of the finite topology. We then give some connections between density properties, direct sum decompositions and the exactness of Rat. We show 
(see \coref{2.6}) that the
Rat functor is exact if the coring $\Cc$
can be decomposed as a direct sum of finitely generated left $\Cc$-comodules.
Under certain conditions, which hold if $R$ is a qF-ring, we can prove the converse,
namely if Rat is exact, then there is a direct sum decomposition of $\Cc$ into
finitely generated comodules. This is in fact an application of the duality
between left and right finitely generated modules over qF-rings.\\ 
In \seref{3}, we characterize semiperfect corings over artinian rings. The main result
is \thref{3.1}, stating that a coring over an artinian ring is right semiperfect if and only
if the category of right comodules has enough projectives, if and only if it has
a projective generator, if and only if every finitely generated comodule has a finitely generated projective cover.\\
In \seref{5}, we discuss some applications and examples. First, we apply our
results to the case where $R$ is a qF-ring. We recover a result of \cite{Kaoutit3}
telling that a left and right (locally) projective coring over a qF-ring is
right semiperfect if and only if the Rat functor is exact. Also two-sided
prefectness is equivalent to two-sided semiperfectness for corings over qF-rings.\\
finally, we give some examples, focussing on the Sweedler coring associated
to a ring morphism. In particular, we can describe the Rat functor in this situation,
and we can discuss when the assumptions of the results in Section \ref{se:3}
and \ref{se:5.1} are satisfied.

\section{Preliminary results}\selabel{1}
\subsection{The category $\sigma[M]$}\selabel{1.1}
Let $R$ be a ring, and $M\in {}_R\Mm$. Recall from \cite[Sec. 15]{W}
that $\sigma[M]$ is the full subcategory
of ${}_R\Mm$ consisting of $R$-modules that are subgenerated by $M$,
that is, submodules of an epimorphic image of $M^{(I)}$, for some
index set $I$. $\sigma[M]$ is the smallest closed subcategory of
${}_R\Mm$ containing $M$. Since epimorphic images of objects of $\sigma[M]$
belong to $\sigma[M]$ (see \cite[Prop. 15.1]{W}), we have for any
$N\in {}_R\Mm$ that
$$\Tt^M(N)=\sum \{f(X)~|~X\in \sigma[M],~f\in {}_R\Hom(N,X)\}\in \sigma[M].$$
$\Tt^M:\ {}_R\Mm\to \sigma[M]$ is called the trace functor, and it is straightforward
to show that $\Tt^M$ is the right adjoint of the inclusion functor
$i:\ \sigma[M]\to {}_R\Mm$. Therefore $\Tt^M$ is left exact; it is also not difficult
to see that
$$\Tt^M(N)=\sum \{X~|~X\subset \sigma[M],~X\subset M\}.$$
For $X,Y\in {}_R\Hom(X,Y)$, we consider the finite topology on 
${}_R\Hom(X,Y)$. A basis of open sets consists of
$$\Oo(f,x_1,\cdots, x_n)=\{g\in {}_A\Hom(X,Y)~|~
g(x_i)=f(x_i),~{\rm for~all}~i=1,\cdots,n\}$$
We have a natural map $r:\ R\to {}_\ZZ\Hom(M,M)$, $r_a(m)=am$. The
finite topology on ${}_ZZ\Hom(M,M)$ induces a topology on $R$, called the
$M$-adic topology.\\
An ideal $T$ of $R$ is called $M$-dense in $R$ if it is dense in the
$M$-adic topology. This means that for all $a\in R$ and $m_1,\cdots,m_n\in M$,
there exists a $b\in T$ such that $am_i=bm_i$, for all $i$.
A left $T$-module $N$ is called unital if for every
$n\in N$, there exists $t\in T$ such that $tn=n$, or, equivalently,
for every finite $\{n_1,\cdots,n_k\}\subset N$, there exists $t\in N$
such that $tn_i=n_i$, for all $i$.\\
The proof of \prref{1.1} is straightforward; we also refer to \cite[Sec. 41]{BW}.

\begin{proposition}\prlabel{1.1}
Let $R$ be a ring, and $M\in {}_R\Mm$.\\
(a) For an ideal $T$ of $R$, and a faithful $R$-module $M$, the
following assertions are equivalent.
\begin{itemize}
\item[(i)] $T$ is $M$-dense in $R$;
\item[(ii)] $M$ is a unital $T$-module (with the induced structure from $R$);
\item[(iii)] $TN=N$ for all $N\in \sigma[M]$;
\item[(iv)] the multiplication map $T\ot_R N\to N$ is an isomorphism.
\end{itemize}
(b) $T=\Tt^M(A)$ is an ideal of $A$, and the following assertions are
equivalent.
\begin{itemize}
\item[(i)] $T$ is $M$-dense in $A$;
\item[(ii)] $M$ is a $T$-unital module;
\item[(iii)] $\Tt^{M}$ is exact;
\item[(iv)] $T^{2}=T$ and $T$ is a generator in $\sigma[M]$.
\end{itemize}
\end{proposition}

Let $K$ be an $A$-submodule of $M$. Recall (see e.g. \cite[19.1]{W}) that
$K$ is called superfluous or small, written $K\ll M$, if for every
submodule $L\subset M$, $K+L=M$ implies that $L=M$. An epimorphism
$f:\ M\to N$ is called superfluous if $\Ker f\ll M$. Note that this definition can be
extended to abelian categories.

\begin{proposition}\prlabel{1.2}
Assume that $\Tt^{M}$ is exact.
\begin{itemize}
\item[(i)] The class $\sigma[M]$ is closed under small epimorphisms in ${}_{A}\Mm$;
\item[(ii)] the inclusion functor $\sigma [M]\to {}_{A}\Mm$ preserves projectives.
\end{itemize}
\end{proposition}

\begin{proof}
(i) Take $N\in \sigma[M]$, and let
$$0\to K\to X\to N$$
be an exact sequence in ${}_A\Mm$ such that $K$ is small in $X$.
then $Y=X/(K+\Tt^M(X))$ is a quotient of $X/K=N\in \sigma[M]$, so $Y\in  \sigma[M]$,
by \cite[15.1]{W}, and $\Tt^M(Y)=Y$. Consider the exact sequence
$$0\to \Tt^M(X)\to X\to X/\Tt^M(X)\to 0.$$
Since $\Tt^M$ is exact and idempotent, it follows that $\Tt^M(X/\Tt^M(X))=0$.
Now $Y$ is a quotient of $X/\Tt^M(X)$, and it follows from the exactness of
$\Tt^M$ that $\Tt^M(Y)=0$. Thus $Y=0$, and $K+\Tt^M(X)=X$. Since $K\ll X$,
we have that $\Tt^M(X)=X$, so $X\in \sigma[M]$, as needed.
\end{proof}

\subsection{Properties of the finite topology}\selabel{1.2}
\begin{proposition}\prlabel{1.3}
Let $R$ be a ring, and fix  a right $R$-module $T$.
Density will mean density in the finite topology.
\begin{enumerate}
\item[(i)] Let $M=M_1\oplus M_2$ in $\Mm_R$, and 
$X_{1}\subset \Hom_R(M_1,T)$, $X_{2}\subset \Hom_R(M_2,T)$
If $X_{1}\oplus X_{2}$ is dense in $\Hom_R(M,T)=\Hom_R(M_1,T)\oplus \Hom_R(M_1,T)$, then each $X_{i}$ is dense in $\Hom_R(M_i,T)$.
\item[(ii)] Let $(M_{i})_{i\in I}$ be a family of $R$-modules, and 
$X_{i}\subset \Hom_R(M_i,T)$ such that each $X_{i}$ is dense in $\Hom_R(M_i,T)$. Let
$M=\bigoplus_{i\in I}M_{i}$. 
Then $\bigoplus  _{i\in I} X_{i}$ is dense in 
$\Hom_R(M,T)=\prod _{i\in I} \Hom_R(M_i,T)$.
\end{enumerate}
\end{proposition}

\begin{proof}
(i) Take $f\in\Hom_{R}(M_{1},T)$ and $F$ is a finite subset of $M_{1}$. Viewing $f$ as the pair $(f,0)\in \Hom_R(M_1,T)\oplus \Hom_R(M_2,T)$ and $F\subset M_{1}\subset M_{1}\oplus M_{2}$,
 we find a pair $(g,h)\in X_{1}\oplus X_{2}\subset
\Hom_R(M,T)=\Hom_R(M_1,T)\oplus \Hom_R(M_2,T)$ such that $(g,h)=(f,0)$ on $F$, so $g=f$ on all $m\in F$, with $g\in X_{1}\subset \Hom_R(M_1,T)$. \\
(ii) Take $(f_{i})_{i\in I}\in \Hom_R(M,T)=\prod _{i\in I} \Hom_R(M_i,T)$ and 
 a finite subset $F\subset \bigoplus_{i\in I}M_{i}$. Then there is a finite subset $J\subset I$ such that $F\subset \bigoplus_{i\in J}M_{i}$. 
$F_i=\{m_i~|~m\in F\}$ is finite, and, using the density of $X_i$ in
$\Hom_R(M_i,T)$, we find $g_i\in X_i$ such that $g_i=f_i$ on $F_i$. Now let
$g\in \prod _{i\in I} \Hom_R(M_i,T)=\Hom_{R}(M,T)$ be defined as follows: the $i$-th
component of $g$ is $g_i$ if $i\in J$, and it is zero otherwise.
Then $g\in \bigoplus_{i\in I}X_{i}$ and $g=f$ on all $F_{i}$,
and a fortiori on $F$, by linearity.
\end{proof}

\begin{corollary}\colabel{1.4}
If $(M_{i})_{i\in I}$ is a family of $R$-modules and 
$X_{i}\subset \Hom_R(M_i,T)$ then $\bigoplus_{i\in I}X_{i}$ is dense in $\prod_{i\in I}\Hom_R(M_i,T)=\Hom_R(\bigoplus_{i\in I}M_{i}, T)$ if and only if all $X_{i}$ are dense in $\Hom_R(M_i,T)$. Consequently, the direct sum $\bigoplus_{i\in I}\Hom_R(M_i,T)$ is dense in the direct product $\prod_{i\in I} \Hom_R(M_i,T)$.
\end{corollary}

\begin{proposition}\prlabel{1.5}
Let $T\in \Mm_R$ be an injective module, and $u:\ X\to Y$ a monomorphism in
$\Mm_R$. If $V$ is dense in $\Hom_R(Y,T)$, then $\Hom_R(u,T)(V)$ is dense in
$\Hom_R(X,T)$.
\end{proposition}

\begin{proof}
Take $f\in \Hom_R(X,T)$, and a finite subset $F\subset X$. As $T$ is an injective module, we can find $g\in \Hom_R(Y,T)$ such that $g\circ u=f$. As $u(F)$ is a finite subset of $Y$ we can find $h\in V$ such that $h$ equals $g$ on $u(F)$. Now we obviously have that $\Hom_R(u,T)(h)=h\circ u$ equals $g\circ u=f$ on $F$, hence $\Hom_R(u,T)(V)$ is dense in $\Hom_R(X,T)$.
\end{proof}

\section{Corings and the Rat functor}\selabel{2}
\subsection{Corings}\selabel{2.1}
Let $R$ be a ring. An $R$-coring is a coalgebra in the monoidal category ${}_R\Mm_R$.
It consists of a triple $\Cc=(\Cc,\Delta,\varepsilon)$, where
$\Cc$ is an $R$-bimodule, and $\Delta:\ \Cc\to \Cc\ot_R\Cc$ and
$\varepsilon:\ Cc\to R$ are $R$-bimodule maps satisfying appropriate
coassociativity and counit properties. We refer to \cite{B}
and \cite{BW} for more detail about corings.
We use the Sweedler-Heyneman notation
$$\Delta(c)=c_{(1)}\ot_R c_{(2)},$$
where the summation is implicitely understood.
If $\Cc$ is an $R$-coring, then ${}^*\Cc={}_R\Hom(\Cc,R)$ is a ring
with multiplication given by the formula
$$(f\#g)(c)=g(c_{(1)}f(c_{(2)})).$$
The unit of the multiplication is $\varepsilon$. We have a ring morphism
$$\iota:\ R\to {}^*\Cc,~\iota(r)(c)=\varepsilon(c)r.$$
A right $\Cc$-comodule consists of a pair $(M,\rho^r)$, where $M\in \Mm_R$
and $\rho^r:\ M\to M\ot_R \Cc$ is a right $A$-linear map satisfying
the conditions
$$(\rho^r\ot_R \Cc)\circ \rho^r= (M\ot_R\Delta)\circ\rho^r~~{\rm and}~~
(M\ot_R\varepsilon)\circ \rho^r=M.$$
Left $\Cc$-comodules are defined in a similar way, and the categories
of left and right $\Cc$-comodules are respectively denoted by
$\Mm^\Cc$ and ${}^\Cc\Mm$. We use the Sweedler-Heyneman notation
$$\rho^r(m)=m_{[0]}\ot_R m_{[1]}~~{\rm and}~~\rho^l(m)=m_{[-1]}\ot_R m_{[0]}$$
for right and left $\Cc$-coactions. We have a functor
$F:\ \Mm^\Cc\to \Mm_{{}^*\Cc}$, with
$F(M)=M$ as an $R$-module, equipped with the right ${}^*\Cc$-action
$m\cdot f=m_{[0]}f(m_{[1]})$. In particular, $\Cc$ is a right and left
${}^*\Cc$-module. If $M$ and $N$ are right $\Cc$-comodules, then the
set of $R$-linear maps preserving the $\Cc$-coaction is denoted
by $\Hom^\Cc(M,N)$.

\subsection{The $\alpha$-condition}\selabel{2.2}
$M\in {}_R\Mm$ satisfies the (left) $\alpha$-condition if the canonical map
$$\alpha_{N,M}:\ N\ot_R M\to \Hom_R({}^*M, N),~~
\alpha(n\ot_R m)(f)=nf(m)$$
is injective, for all $N\in \Mm_R$. Otherwise stated: if $n\ot_R m\in
N\ot_R M$ is such that $nf(m)=0$ for all $f\in {}^*M$, then $n\ot m=0$.
$M$ satisfies the $\alpha$-condition if and only if $M$ is locally
projective in ${}_R\Mm$. An $R$-coring $\Cc$ satisfies the left
$\alpha$-condition if and only if 
$\Mm^\Cc$ is a full subcategory of $\Mm_{{}^*\Cc}$, and
the natural functor $\Mm^\Cc\to \sigma[C_{{}^*\Cc}]$ is an isomorphism.
In this case, $\Cc$ is flat as a left $R$-module, hence $\Mm^\Cc$ is
a Grothendieck category in such a way that the forgetful functor
$\Mm^\Cc\to \Mm_A$ is exact (see \cite[Sec. 19]{BW}).\\

If $\Cc\in {}_R\Mm$ is locally projective, then for all $M\in \Mm^\Cc$,
the lattices consisting respectively of all $\Cc$-subcomodules
and of all  ${}^{*}\Cc$-submodules of $M$ coincide, so it makes
sense to talk about the subcomodule generated by a subset of $M$.
>From the proof of \cite[19.12]{BW}, we deduce the following result.

\begin{theorem}\thlabel{finite} {\bf (Finiteness Theorem)}
If $\Cc\in {}_R\Mm$ is locally projective, then a right $\Cc$-comodule $M$ is
finitely generated as a right $\Cc$-comodule if and only if it is finitely
generated as a right $R$-module.
\end{theorem}

Let $\Cc$ be locally projective as a left $R$-module, and $M$
a right ${}^*\Cc$-module. $\Rat^\Cc(M)$ is by definition the 
largest ${}^*\Cc$-submodule $N$ of $M$, on which there exists a
right $\Cc$-coaction $\rho$ such that $F(N,\rho)=N$. Otherwise
stated, $\Rat^\Cc$ is the preradical functor $\Tt^\Cc$, with $\Cc$
considered as a right  ${}^*\Cc$-module. We also have that
$\Rat^\Cc(M)$ consists of the elements $m\in M$ such that there exists
$m_{[0]}\ot_R m_{[1]}\in M\ot_R \Cc$ with $m\cdot f=m_{[0]}f(m_{[1]})$,
for all $f\in {}^*\Cc$. In a similar way, we define the left Rat functor
${}^\Cc\Rat$. The proof of \prref{2.1} is straightforward, and left to
the reader.

\begin{proposition}\prlabel{2.1}
Let $\Cc$ be an $R$-coring, and $M\in {}^\Cc\Mm$.
\begin{enumerate}
\item[(i)] The $R$-modules ${}^\Cc\Hom(M,\Cc)$ and ${}^*M={}_R\Hom(M,R)$
are isomorphic;
\item[(ii)] ${}^\Cc\Hom(M,\Cc)$ is a right ${}^*\Cc$-module, via
$$(\varphi\cdot f)(m)=f(\varphi(m));$$
\item[(iii)] we have isomorphic functors ${}^\Cc\Hom(-,\Cc)$ and ${}_R\Hom(-,R)$
from ${}^\Cc\Mm$ to $\Mm_{{}^*\Cc}$; these functors are left exact if
$\Cc$ is locally projective in $\Mm_R$, and exact if $R$ is injective as
a left $R$-module;
\item[(iv)] the isomorphism from (i) defines a ring isomorphism
${}^\Cc\End(\Cc)\cong {}^*\Cc$, where the multiplication on ${}^\Cc\End(\Cc)$
is the oppositie composition;
\item[(v)] ${}^\Cc\Hom(M,\Cc)$ is a right ${}^\Cc\End(\Cc)$-module, via
$$(\varphi\cdot f)(m)=f(\varphi(m)).$$
\end{enumerate}
\end{proposition}

Observe that the right coactions defined in (ii) and (v) are the same after
we identify ${}^\Cc\End(\Cc)$ and ${}^*\Cc$ using (iv).\\
Let ${}^{{\rm fg}\Cc}\Mm$ be the category of finitely generated left
$\Cc$-comodules. If $R$ is left noetherian, then the kernel of a morphism
in ${}^{{\rm fg}\Cc}\Mm$ is still finitely generated, hence ${}^{{\rm fg}\Cc}\Mm$
has kernels (and cokernels), and is an abelian category.

\begin{proposition}\prlabel{2.2}
Let $R$ be a left noetherian ring, and $\Cc$ a locally projective $R$-coring.
\begin{enumerate}
\item[(i)] For any finitely generated $M\in {}_R\Mm$, the evaluation map
$$\psi_M:\ {}_R\Hom(M,R)\ot \Cc\to {}_R\Hom(M,\Cc),~~\psi_M(f\ot c)(m)=f(m)c$$
is an isomorphism.
\item[(ii)] Let $(M,\rho_M)\in {}^{{\rm fg}\Cc}\Mm$ and consider the map
$$\phi_M:\ {}^*M\to {}^*M\ot_R\Cc,~~\phi_M(f)=\psi_M^{-1}((\Cc\ot f)\circ \rho_M)$$
Then $({}^*M,\phi_M)\in \Mm^\Cc$, and the associated ${}^*\Cc$-module structure
is as defined in \prref{2.1}
\end{enumerate}
\end{proposition}

\begin{proof}
(i) It is straightforward to prove the statement for free modules. Then we 
can easily show it for finitely presented modules, using the flatness of
$\Cc$ over $R$. Since $R$ is noetherian, every finitely presented module is
finitely generated.\\
(ii) Take $f\in {}^*M$, and write
$$\phi_M(f)=f_{[0]}\ot f_{[1]}\in {}^*M\ot_R\Cc.$$
Then $m_{[-1]}f(m_{[0]})=f_{[0]}(m)f_{[1]}$, and for very ${}^*c\in {}^*\Cc$,
we find that
\begin{eqnarray*}
(f\cdot {}^*c)(m)&=& {}^*c(m_{[-1]}f(m_{[0]}))={}^*c(f_{[0]}(m)f_{[1]})\\
&=& f_{[0]}(m){}^*c(f_{[1]})= (f_{[0]}\cdot {}^*c(f_{[1]})(m)
\end{eqnarray*}
This shows that ${}^*M$ is a rational ${}^*\Cc$-module, and that $\phi_M$
is a right $\Cc$-coaction.
\end{proof}

\subsection{The Rat functor}\selabel{2.3}

Assume that $\Cc$ is a coring satisfying the left $\alpha$-condition.
Then the functor $\Rat^\Cc$ is additive and left exact.

\begin{proposition}\prlabel{2.3}
The following assertions are equivalent.
\begin{enumerate}
\item[(i)] $\Rat^{\Cc}({}^{*}\Cc)$ is dense in ${}^{*}\Cc$ in the $\Cc$-adic topology;
\item[(ii)]  $\Rat^{\Cc}({}^{*}\Cc)$ is dense in ${}^{*}\Cc$ in the finite topology;
\item[(iii)]  $\Rat^{\Cc}$ is an exact functor.
\end{enumerate}
\end{proposition}

\begin{proof}
The equivalence of (i) and (iii) follows from \prref{1.1}, invoking the fact
that $\Cc$ is faithful as a right ${}^*\Cc$-module.\\
Note that the sets 
$${\mathcal O}_{a}(F)=\{{}^{*}c~|~ c\cdot {}^{*}c=0,~{\rm for~all}~ c\in F\},$$
with $F\subset \Cc$ finite, form a basis of open neighborhoods of $0\in \Cc$ in the $\Cc$-adic topology, which is a linear topology. Also
$${\mathcal O}_{f}(F)=\{{}^{*}c~|~ {}^{*}c(c)=0,~{\rm for~all}~c\in F\},$$
with $F\subset \Cc$ finite, form a basis of open neighborhoods of $0$ for the finite topology,
which is also linear.\\
Let $F\subset \Cc$ be finite. For each $c\in F$, we fix a tensor representation of $\Delta(c)$,
and then consider the finite set $F'$ of all second tensor components. Then we easily see
that
$${\mathcal O}_{f}(F^{\prime})\subseteq {\mathcal O}_{a}(F)\subseteq {\mathcal O}_{f}(F)$$
and it follows that the two linear topologies on ${}^*\Cc$ coincide, so it follows that
(i) is equivalent to (ii).
\end{proof}

\begin{proposition}\prlabel{2.4}
Suppose we have a decomposition 
$\Cc=\bigoplus_{i\in I}C_{i}$ as left $\Cc$-comodules. Then $\Rat^{\Cc}({}^{*}C_{i})$ is dense in ${}^{*}C_{i}$ for all $i\in I$ if and only if $\Rat^{\Cc}({}^{*}\Cc)$ is dense in ${}^{*}\Cc$.
\end{proposition}

\begin{proof}
Assume that each $\Rat^{\Cc}({}^{*}C_{i})$ is dense in ${}^{*}C_{i}$. It follows from
\prref{1.3} that $\bigoplus_{i\in I}\Rat^{\Cc}({}^{*}C_{i})$ is dense in ${}^{*}C$,
and then $\Rat^{\Cc}({}^{*}C)\supset \bigoplus_{i\in I}\Rat^{\Cc}({}^{*}C_{i})$ is also
dense.\\
Conversely, let $M=\bigoplus_{j\in I, j\neq i}C_{j}$, for each $i\in I$. Then
$C=C_{i}\oplus M$ and ${}^{*}C={}^{*}C_{i}\oplus {}^{*}M$, hence $\Rat^{\Cc}({}^{*}\Cc)=\Rat^{\Cc}({}^{*}C_{i})\oplus Rat^{\Cc}({}^{*}M)$ is dense in ${}^{*}C={}^{*}C_{i}\oplus {}^{*}M$ ($\Rat^{\Cc}$ is an aditive functor). The result then
follows from \prref{1.3} (ii).
\end{proof}

\begin{lemma}\lelabel{2.5}
\begin{enumerate}
\item[(i)] Assume that $M\in {}^\Cc\Mm$ is finitely generated and projective as a left $R$-module.
Then ${}^*M$ is a rational right ${}^*\Cc$-module.
\item[(ii)] Suppose that $\Cc=M\oplus N$ in ${}^\Cc\Mm$. Then ${}^*M$ is rational if and only if $M$ is finitely generated as a left $R$-module.
\end{enumerate}
\end{lemma}

\begin{proof}
(i) We take a finite dual basis $\{(x^i,f_i)~|~i=1,\cdots,n\}$ of $M\in {}_R\Mm$.
For all $h\in {}^*M$ and $\alpha\in {}^*\Cc$, we have
$$h\cdot\alpha=\sum_{i}f_{i}\cdot (h\cdot\alpha)(x^{i})= \sum_{i}f_{i}\alpha(x^{i}_{[-1]}h(x^{i}_{[0]}))$$
This shows that $h_{[0]}\otimes h_{[1]}=f_{i}\otimes x^{i}_{[-1]}h(x^{i}_{[0]})\in {}^{*}M\otimes \Cc$
is such that 
$h\cdot\alpha=h_{[0]}\alpha(h_{[1]})$, and this proves that ${}^{*}M$ is rational.\\
(ii) One direction follows from (i). Conversely, assume that ${}^{*}M$ is rational.
Take $e=\varepsilon_{\mid_{M}}\in {}^*M$.
We can identify ${}^{*}\Cc={}^{*}M\oplus{}^{*}N$ as right ${}^{*}\Cc$ modules.
For $h\in {}^{*}M$ and $c\in \Cc$, $(e\cdot h)(c)=h(c_{(1)}e(c_{(2)}))=h(c_{(1)}\varepsilon(c_{(2)}))=h(c)$ if $c\in M$ ($c_{(1)}\otimes c_{(2)}\in C\otimes M$) and $(e\cdot h)(c)=h(c_{(1)}e(c_{(2)}))=0$ if $c\in N$ ($c_{(1)}\otimes c_{(2)}\in C\otimes N$) showing that $e\cdot h=h$ (the $h$ in the $e\cdot h$ is regarded as belonging to ${}^{*}\Cc$). As ${}^{*}M$ is rational there is $\sum_{i}f_{i}\otimes x^{i}\in {}^{*}M\otimes \Cc$ such that $e\cdot \alpha=\sum_{i}f_{i}\alpha(x^{i})$, for all $\alpha\in {}^{*}\Cc$. Then for any $h\in {}^{*}M$, $h=e\cdot h=\sum_{i}f_{i}h(x^{i})$,
and, for all $m\in M$, we have
$h(m)=\sum_{i}f_{i}(m)h(x^{i})=h(\sum_{i}f_{i}(m)x^{i})= h(\sum_{i}f_{i}(m)m^{i})$, where $x^{i}=m^{i}+n^{i}\in M\oplus N$ is the unique representation of $x^{i}$ in the direct sum $\Cc=M\oplus N$ and the last equality holds as $h_{\mid_{N}}=0$.   As this last equality holds for all $h\in {}^{*}M$, we can easily see that it actually holds for all $\alpha=(h,g)\in {}^{*}\Cc={}^{*}M\oplus{}^{*}N$ because $m\in M$, and so we now obtain, using the left $\alpha$-condition on ${}^{*}\Cc$, that $m=f_{i}(m)m^{i}$, where $m\in M$ is arbitrary and $m^{i}\in M$ are fixed. Thus $M$ is finitely generated.
\end{proof}

\begin{corollary}\colabel{2.6}
Assume that $\Cc=\bigoplus_{i\in I}C_{i}$ as left $\Cc$-comodules,
and that each $C_i$ is finitely generated. Then $\Rat^\Cc({}^*\Cc)$
is dense in ${}^*\Cc$, and, equivalently, $\Rat^\Cc$ is an exact functor.
\end{corollary}

\begin{proof}
This is a direct consequence of \prref{1.3} (ii) and \leref{2.5}.
\end{proof}

\begin{example}\exlabel{ratfunctor}
We now present an example of a coring for which we can explicitly
construct the Rat functor. Let $G$ be a group, $k$ a commutative ring,
and $R$ a $G$-graded $k$-algebra. It is well-known that
$\Cc=R\ot kG$ is an $R$-coring. The structure maps are given by the formulas
$$r(s\ot \sigma)t=\sum_{\rho\in G} rst_\rho\ot \sigma\rho;$$
$$\Delta_{\Cc}(s\ot\sigma)=(s\ot\sigma)\ot_R(1\ot \sigma)~~;~~
\varepsilon(s\ot\sigma)=s.$$
Here $t_\rho$ is the homogeneous part of degree $\rho$ of $t$.
Clearly $\Cc=\bigoplus_{\sigma\in G}R\ot\sigma$ decomposes as the direct sum
of finitely generated (free of rank one) left $\Cc$-comodules, hence it follows
from \coref{2.6} that Rat is exact. We will illustrate this, computing Rat.
First observe that
$${}^*\Cc={}_R\Hom(R\ot kG,R)\cong\Hom(kG,R)\cong {\rm Map}(G,R).$$
The multiplication on ${}^*\Cc$ can be transported into a multiplication
on ${\rm Map}(G,R)$. This multiplication is the following. For $f,g:\ G\to R$
and $\tau\in G$:
\begin{equation}\eqlabel{rat1}
(f\#g)(\tau)=\sum_\rho f(\tau)_\rho g(\tau\rho)
\end{equation}
Let $(kG)^*$ be the dual of the group algebra $kG$, with free basis
$\{v_\sigma~|~\sigma\in G\}$, such that $v_{\sigma}(\tau)=\delta_{\sigma,\tau}$.
then $v_\sigma$ can also be viewed as a map $G\to R$, and this gives us an
algebra embedding $(kG)^*\subset {\rm Map}(G,R)$. Indeed, using \equref{rat1},
we easily compute that $v_\sigma\#v_\tau=\delta_{\sigma,\tau}v_{\sigma}$.\\
We also have an algebra embedding
$$\iota:\ R\to {\rm Map}(G,R),~~\iota_r(\sigma)=r.$$
Indeed, using \equref{rat1}, we find
$$(\iota_r\#\iota_s)(\tau)=\sum_\rho \iota_r(\tau)_\rho\iota_s(\tau\rho)=
\sum_{\rho}r_\rho s=rs=\iota_{rs}(\tau).$$
Let $r\in R$ be homogeneous of degree $\rho$, and $f:\ G\to R$. Using
\equref{rat1}, we compute
\begin{equation}\eqlabel{rat2}
v_\sigma\#\iota_r=\iota_r\# v_{\sigma\rho}~~{\rm and}~~
v_\sigma\# f=v_\sigma\# \iota_{f(\sigma)}.
\end{equation}
Now take $M\in \Mm_{{}^*\Cc}\cong \Mm_{{\rm Map}(G,R)}$. By restriction of
scalars, $M$ is also a right $R$-module and a right $(kG)^*$-module. Now
put $M_{\sigma}=M\cdot v_{\sigma}$.\\
1) If $\sigma\neq\tau$, then $M_\sigma\cap M_\tau=0$. Indeed, if
$m\cdot v_{\sigma}=n\cdot v_\tau$, then
$$m\cdot v_{\sigma}=m\cdot (v_{\sigma}\#v_\sigma)=(m\cdot v_{\sigma})\cdot v_{\sigma}=
(n\cdot v_{\tau})\cdot v_{\sigma}=n\cdot (v_{\tau}\#v_\sigma)=0.$$
2) $M_\sigma R_\rho\subset M_{\sigma\rho}$. Take $m\cdot v_\sigma\in M_\sigma$
and $r\in R_\rho$. Using \equref{rat2}, we find
$$(m\cdot v_\sigma)r=m\cdot (v_\sigma\#\iota_r)=m\cdot (\iota_r\#v_{\sigma\rho})=
(mr)\cdot v_{\sigma\rho}\in M_{\sigma\rho}.$$
This shows that $\bigoplus_{\sigma\in G}M_\sigma$ is a $G$-graded $R$-module; we will
show that it is the rational part of $M$.\\
3) $M_\sigma\subset \Rat(M)$. Take $m\cdot v_\sigma\in M_\sigma$ and $f\in
{\rm Map}(G,R)$. Using \equref{rat2}, we find
$$(m\cdot v_\sigma)\cdot f=m\cdot (v_\sigma\# f)=m\cdot (v_\sigma\# \iota_{f(\sigma)})=
(m\cdot v_\sigma)f(\sigma),$$
so $m\cdot v_\sigma$ is rational.\\
4) It follows from 3) that $\bigoplus_{\sigma\in G}M_\sigma\subseteq \Rat(M)$.\\
5) Let $m\in \Rat(M)$. Then there exist $m_1,\cdots,m_n\in M$,
$r_1,\cdots,r_n\in R$ and $\sigma_1,\cdots,\sigma_n\in G$ such that,
for all $\varphi\in {}^*\Cc$:
$$m\cdot\varphi=\sum_i m_i\varphi(r_i\ot \sigma_i).$$
Making the identification ${}^*\Cc\cong {\rm Map}(G,R)$, we find for all
$f:\ G\to R$:
$$m\cdot f=\sum_i m_ir_if(\sigma_i).$$
Replacing $m_i$ by $m_ir_i$, it is no restriction to take $r_i=1$. We can also
take the $\sigma_i$ pairwise different. Taking
$f=v_{\sigma}$, we find that
$$m_\sigma=\sum_i m_i\delta_{\sigma,\sigma_i}$$
so $m_\sigma\neq 0$ for only a finite number of $\sigma$, and $m_{\sigma_i}=m_i$.
Finally 
$$m=m\cdot \iota_1=\sum_im_i\iota_1(\sigma_i)=\sum_im_i=\sum_i m_{\sigma_i}\in
\bigoplus_{\sigma\in G}M_\sigma.$$
We conclude that
$$\Rat(M)=\bigoplus_{\sigma\in G} M\cdot v_\sigma,$$
and it is clear that $\Rat$ is exact.
\end{example}

In some situations, the converse of
\coref{2.6} also holds. If $R$ is left artinian, then any left comodule 
contains a simple comodule. The same holds for comodules that are locally 
artinian, in the sense that any finitely generated submodule is artinian.
If this is the case for $\Cc$, then the left socle of $C$ is essential in $\Cc$.
If moreover $\Cc$ is injective in ${}^{\Cc}\Mm$, then a decomposition 
$\Cc=\bigoplus_{i\in I} E(S_i)$ holds with usual arguments, where 
$\bigoplus_{i\in I}S_{i}={}^{\Cc}s(\Cc)$ is a decomposition of the 
left socle ${}^{\Cc}s(\Cc)$ of $\Cc$ and $E(S_{i})$ is the injective hull of $S_{i}$ contained in
$\Cc$. We will assume that $\Cc$ is locally projective
as a right $R$-module, which implies that ${}^{\Cc}\Mm$ is abelian, so that
we have a categorical definition of injective hulls.

\begin{proposition}\prlabel{2.7}
Assume that $\Cc$ also satisfies the right $\alpha$-condition, and that the
two following conditions hold:
\begin{enumerate}
\item $\Cc$ is an injective object of ${}^\Cc\Mm$;
\item $R$ is left artinian or
$\Cc$ is locally artinian in ${}_R\Mm$ (equivalently in ${}^\Cc\Mm$).
\end{enumerate}
Let $\bigoplus_{i\in I} S_i$ be the decomposition of the left socle of $\Cc\in
{}^\Cc\Mm$ into simple left $\Cc$-comodules, 
and $E(S_i)$ an injective envelope of $S_i$ contained in $\Cc$.
Then $\Rat^\Cc$ is exact if and only if each $E(S_i)$ is finitely generated.
\end{proposition}

\begin{proof}
We have that $\Cc=\bigoplus_{i\in I} E(S_i)$, so one direction follows from
\coref{2.6}. Conversely, assume that $\Rat^\Cc$ is exact, and
let $S$ be a simple subcomodule of $C$, and
$E(S)$ an injective envelope of $S$ contained in $C$. 
Then there is a left subcomodule $X$ of $\Cc$ such that $E(S)\oplus X=\Cc$ in ${}^\Cc\Mm$. The functor ${}^\Cc\Hom(-,\Cc)$ is exact since $\Cc\in \Mm^\Cc$
is injective, and the composition of ${}^\Cc\Hom(-,\Cc)$ with the natural
functor $\Mm^\Cc\to \Mm_{{}^*\Cc}$ is also exact. Thus we obtain an
epimorphism $\pi:\ {}^*E(S)\to {}^*S$, with kernel
${}^{\perp}S=\{ f\in {}^{*}E(S)~|~ f_{\mid}{}_{S}=0\}$.\\ 
We will first show that ${}^{\perp}S\ll {}^{*}E(S)$. Using the isomorphisms
in \prref{2.1}, we can regard $\pi$ as a left ${}^\Cc\End(\Cc)$-module morphism
${}^\Cc\Hom(E(S),\Cc)\to {}^\Cc\Hom(S,\Cc)$.\\
Take $f\in {}^\Cc\Hom(E(S),\Cc)\setminus {}^\perp S$, i.e.
$f:\ E(S)\to \Cc$ such that $f_{|S}\neq 0$. Then $\Ker f\cap S=0$
since $S$ is simple, and therefore $\Ker f=0$, since $S$ is essential in
$E(S)$. So $E(S)\cong f(E(S))$, and there exists a left
$\Cc$-subcomodule $M$ of $\Cc$ such that $\Cc\cong f(E(S))\oplus M$.
We can extend $f$ to a left $\Cc$-comodule isomorphism
$\ol{f}:\ \Cc\to \Cc$, since $X\cong M$. Let $h$ be the inverse of $\ol{f}$.
Take an arbitrary
$g\in  {}^\Cc\Hom(E(S),\Cc)$ ,and extend $g$ to
$\ol{g}:\ \Cc= E(S)\oplus X\to \Cc$ by putting $\ol{g}_{|X}=0$.
Then $\ol{g}=\ol{g}\circ h\circ \ol{f}$, which means that
${}^\Cc\Hom(E(S),\Cc)$ is generated by $\ol{f}$ as a left
${}^\Cc\End(\Cc)$-module. Consequently 
 ${}^{\perp}S\ll {}^{*}E(S)$. \\
The Finiteness \thref{finite} shows that $S$ is finitely generated 
and then it follows from \prref{2.2} (ii) that
${}^{*}S$ is a rational ${}^{*}C$-comodule, so $\Rat^{\Cc}({}^{*}S)={}^{*}S$. 
$\Rat^{\Cc}$ is exact, so we have an exact sequence 
$$0\longrightarrow \Rat^{\Cc}({}^{\perp}S)\longrightarrow \Rat^{\Cc}({}^{*}E(S)) \stackrel{\pi}{\longrightarrow} \Rat^{\Cc}({}^{*}S)={}^{*}S\longrightarrow 0. $$
We obtain $\pi (\Rat^{\Cc}({}^{*}E(S)))={}^{*}S$, so ${}^{\perp}S+\Rat^{\Cc}({}^{*}E(S))={}^{*}E(S)$. It then follows that
${}^{*}E(S)$ is rational. This last part can also be seen as follows.
We have an exact sequence
$$0\longrightarrow{}^\perp S\longrightarrow {}^*E(S)\longrightarrow {}^*S
\longrightarrow 0,$$
with ${}^\perp S\ll {}^*E(S)$ and ${}^*S$ rational, so ${}^*E(S)$ is
rational by \prref{1.2}(i). 
Using \leref{2.5}, we find that ${}_{R}E(S)$ is finitely generated.
\end{proof}

\section{Semiperfect corings}\selabel{3}
Let $\Cc$ be an abelian category. A projective object $P\in \Cc$
together with a superfluous epimorphism $P\to M$ is called a projective
cover of $M$. $\Cc$ is called semiperfect if every simple
object has a projective cover. If a coring $\Cc$ satisfies the left
$\alpha$-condition, then $\Mm^\Cc$ is an abelian category, and $\Cc$
is called right semiperfect if $\Mm^\Cc$ is semiperfect. Semiperfect corings
were introduced first in \cite{Kaoutit3}.

\begin{theorem}\thlabel{3.1}
Let $R$ be a right artinian ring, and $\Cc$ an $R$-coring satisfying the left $\alpha$-condition.
The following statements are equivalent.
\begin{itemize}
\item[(i)] $\Cc$ is right semiperfect;
\item[(ii)]  Every finitely generated right comodule has a projective cover;
\item[(iii)]  every finitely generated right comodule has a finitely generated projective cover;
\item[(iv)]  the category $\Mm^{\Cc}$ has enough projectives;
\item[(v)]  every simple right comodule has a finitely generated projective cover; 
\item[(vi)]  the category $\Mm^\Cc$ has a progenerator (=projective generator).
\end{itemize}
\end{theorem}

\begin{proof}
(i)$\Rightarrow$(ii). First notice that an $R$-module is finitely generated if and only if
it has finite length. Every finitely generated comodule $M$ has a maximal
subcomodule, so its Jacobson  radical $J(M)$ in $\Mm^\Cc$
is different from the comodule itself. $J(M)\ll M$, and $M/J(M)$ is a semisimple finitely
generated comodule. Every simple component of $M/J(M)$ has a projective
cover, and the direct sum of all these projective covers is a projective cover
$f: P\to M/J(M)$ of $M/J(M)$. Since $P$ is
projective, there exists $g:\ P\to M$ such that $u\circ g= f$, with
$u:\ M\to M/J(M)$ the canonical projection. Then a usual argument shows that $g:P\rightarrow M$ is a projective cover:
 $u(g(P))=f(P)=M/J(M)$, hence $u(J(M)+g(P))=M/J(M)$ and it follows that $J(M)+g(P)=M$.
 From the fact that $J(M)$ is small in $M$, it follows that $g(P)=M$ and $g$ is surjective.
 Finally $\Ker g \subset \Ker f\ll P$, so $\Ker g \ll P$, and $g:\ P\to M$ is a projective cover
 of $M$.\\
 
 (iv)$\Rightarrow$(iii). Let $M$ be a finitely generated comodule. We know that
 there exists a projective object $P\in \Mm^{\Cc}$ and a $\Cc$-colinear epimorphism
 $f:\ P\to M$. Let $(M_i)_{i\in I}$ be a family of finitely generated comodules such that
 we have a $\Cc$-colinear epimorphism $f:\ \bigoplus_{i\in I} M_i\to P$.
 As $P$ is projective, we have that $\bigoplus_{i\in I} M_i\cong P\oplus X$
 as comodules. Since $R$ is artinian, we can assume that the $M_i$ are
 indecomposable. As they have finite length in $\Mm_R$, they also have
 finite length in $\Mm^\Cc$ and $\Mm_{{}^*\Cc}$, so their ${}^*\Cc$-endomorphism rings
 are local, by the Krull-Schmidt Theorem (see \cite[12.8]{AF}). 
 It then follows from the Crawley-J\o nsson-Warfield Theorem (see \cite[26.5]{AF}) that
 $P\cong \bigoplus_{i\in J} M_i$, with $J\subset I$. The $M_i$ are finitely
 generated (rational) ${}^*\Cc$-modules, and are projective objects of
 $\Mm^\Cc$, since they are direct summands of $P$. Since $M$ is finitely
 generated, we can find a finite $F\subset J$ and a projection
 $\bigoplus_{i\in F} M_i\to M$, induced by $f$. Thus we have found
 a finitely generated projective object $P\in {}^\Cc\Mm$ and a $\Cc$-colinear epimorphism
 $f:\ P\to M$. Dualizing the proof of the Eckmann-Schopf Theorem on the
 existence of the injective envelope of a module, see e.g. \cite[18.10]{AF}, we can show that
 $M$ has a projective cover. This works as follows.\\
 $\bullet$ Let $K=\Ker f$, and consider the set $V$ consisting
 of subcomodules $H\subset K$ such that $K/H\ll P/H$, which is equivalent to
 $$H\subset T\subset P, K+T=P~~\Longrightarrow~~T=P$$
 $V\neq \emptyset$ since $K\in V$. $V$ contains a minimal element $K'$ since
 $R$ is artinian.\\
 $\bullet$ Then consider the set $W$ consisting of subcomodules $Y\subset P$ such that $K'+Y=P$.
 This set is nonempty, since $P$ belongs to it. Then take an element in this
 set such that $K'\cap Y$ is minimal. Let $p:\ P\to P'=P/K'$ be the projection. Since $P$ is projective, there
 exists a comodule morphism $h:\ P\to Y$ such that 
$p_{|Y}\circ h=p$, that is, the following diagram commutes:
$$\begin{diagram}
&&P\\
&\SW^h&\dTo_p\\
Y&\rTo^{p_{|Y}}& P'
\end{diagram}$$
We will now show that $p_{|Y}$ is an isomorphism.\\
$\bullet$ $h$ is surjective. Take $y\in Y$. Then
$$p(y-h(y))=p(y)-p(h(y))=p(y)-p(y)=0$$
so $y-h(y)\in K'$ and
$$y=(y-h(y))+h(y)\in (Y\cap K')+\im h.$$
It follows that $Y\subset (Y\cap K')+\im h$. The converse implication is obvious,
so
$$Y= (Y\cap K')+\im h$$
It then follows that
$$P=Y+K'= (Y\cap K')+\im h +K'=\im h+K'$$
The minimality condition on $Y$ then yields that $Y=\im h$, so $h$ is surjective.\\
$\bullet$ $Y\cap K'\ll Y$. If $H\subset Y$ and $H+(Y\cap K')=Y$, then
$H+K'=H+(Y\cap K')+K'=Y+K'=P$. This means that $H\in W$, and 
 the minimality condition on $Y$ gives us that
$H\cap K'\supset Y\cap K'$, and $H\cap K'\subset Y\cap K'$ since $H\subset Y$.
Then we find that $Y= H+(Y\cap K')=H+(H\cap K')=H$, as needed.\\
$\bullet$ From the fact that $0=p(K')=(p\circ h)(K')$,
it follows that $h(K')\subset \Ker (p_{|Y})=Y\cap K'$.\\
$\bullet$ $\Ker h =K'$. It is clear that $\Ker h \subset K'$.
It follows that $K'\subset \Ker h$ if we can show that $\Ker h\in V$, or
$$\Ker h\subset T\subset P, K+T=P~~\Rightarrow~~T=P$$
Assume $\Ker h\subset T\subset P$ and $K+T=P$. Since $K'\subset P$,
we find that $K+K'+T=P$. Also $K'\subset T+K'\subset P$, so it follows from
the fact that $K'\in V$ that $K'+T=P$.Then $h(K')+h(T)=h(P)=Y$, since $K$ is surjective.
Since $h(K')\subset Y\cap K'$, this implies that $Y\cap K'+h(T)=Y$, hence
$h(T)=Y$, since $Y\cap K'\ll Y$, and finally $T=P$ because $T\subset \Ker h$.\\
$\bullet$ $p_{|Y}$ is surjective, as $p=p_{|Y}\circ h$ and $p$ is an epimorphism. \\
$\bullet$ $p_{|Y}$ is injective. Take $y\in Y$ such that $p(y)=0$. $h$ is
surjective, so $y=h(z)$. Then $0=p(y)=p(h(z))=p(z)$, so $z\in K'=\ker h$,
and $y=h(z)=0$.\\
$\bullet$ It now follows that $Y\cap K'=0$. We know from the definition of $Y$
that $Y+K' =P$. Hence $Y\oplus K'=P$, and $P'\cong Y$ is 
finitely generated projective, being a direct factor of $P$. Now look at the
commutative diagram
$$\begin{diagram}
0&\rTo^{}&K'&\rTo^{}&P&\rTo^{p}&P'&\rTo^{}&0\\
&&\dTo^{\subset}&&\dTo^{=}&&&&\\
0&\rTo^{}&K&\rTo^{}&P&\rTo^{f}&M&\rTo^{}&0
\end{diagram}$$
It follows that we have an epimorphism $P'\to M$ in $\Mm^\Cc$,
with kernel $K/K'$. This is a projective cover, since $K/K'\ll P'=P/K'$. 
Moreover, $P^{\prime}$ is finitely generated as a quotient of $P$.\\

(ii)$\Rightarrow$(vi). Take a family $(M_i)_{i\in I}$ consisting of finitely
generated comodules that generate $\Mm^\Cc$. Let $P_i\to M_i$ be a projective
cover of $M_i$. Then $\bigoplus_{i\in I} P_i$ is a projective generator of
$\Mm^\Cc$.\\

(vi)$\Rightarrow$(iv), (iii)$\Rightarrow$(ii)$\Rightarrow$(i) and (iii)$\Rightarrow$(v)$\Rightarrow$(i) are obvious.
\end{proof}
 
 \begin{proposition}\prlabel{3.2}
 Let $R$ be a right artinian ring, and $\Cc$ an $R$-coring satisfying the left $\alpha$-condition.
 \begin{enumerate}
 \item[(i)] $\Mm^{{\rm fg}\Cc}$ is an abelian category;
 \item[(ii)] $Q\in \Mm^{{\rm fg}\Cc}$ is injective if and only if $Q$ is
 an injective object in $\Mm^\Cc$;
 \item[(iii)] $P\in \Mm^{{\rm fg}\Cc}$ is projective if and only if $P$ is
 a projective object in $\Mm^\Cc$.
 \end{enumerate}
 \end{proposition}
 
 \begin{proof}
 (i) The fact that $\Mm^{{\rm fg}\Cc}$ has kernels follows from the assumption
 that $R$ is right artinian and the Finiteness Theorem.\\
 
 (ii) This is a straighforward adaptation of the corresponding result on
 comodules over a coalgebra. Let $u:\ N\to M$ be a monomorphism in $\Mm^\Cc$
 and $f:\ N\to Q$. Consider the set
 $$X=\{(N',f')~|~N\subset N'\subset M,~f':\ N'\to Q,~f'_{|N}=f\}$$
 ordered by the relation $(N',f')<(N'',f'')$ if $N'\subset N''$ and
 $f''_{N'}=f'$. Take a maximal element $(N_0,f_0)$ in $X$, and assume that
 $N_0\neq M$. Take $m\in M\setminus N_0$ and $X$ the subcomodule of $M$
 generated by $M$. By the Finiteness Theorem for comodules, $X$ is finitely
 generated, so there exists $g:\ X\to Q$ such that the following diagram
 commutes:
 $$\begin{diagram}
 0&\rTo^{}&N_0\cap X&\rTo^{}& X\\
 &&\dTo^{f_{0|N_0\cap X}}&\SW_{g}\\
 &&Q&&
 \end{diagram}$$
 Then consider the map $f':\ N'=N_0+X\to Q$, defined by $f'(n_0+x)=
 f_0(n_0)+g(x)$. The usual computation shows that $f'$ is well-defined,
 and $(N',f')$ is an element in $X$ that is strictly greater than $(N_0,f_0)$,
 a contradiction.
 \begin{equation}\eqlabel{diagram}
 \begin{diagram}
 &&P&&\\
 &\SW^{g}&\dTo_{f'}&&\\
 Y'&\rTo^{\pi}&X'&\rTo^{}0\\
 \dTo_{\subset}&&\dTo_{\subset}&&\\
 Y&\rTo^{\pi}&X&\rTo^{}0
 \end{diagram}
 \end{equation}
 (iii) Let $\pi:\ Y\to X$ and $f:\ P\to X$ be morphisms in $\Mm^\Cc$, with
 $\pi$ surjective. Let $\{p_1,\cdots,p_n\}$ be a set of generators of $P$
 as an $R$-module (and a fortiori as a $\Cc$-comodule). Then $X'=\im f$
 is generated by  $\{x_1,\cdots,x_n\}$, with $x_i=f(p_i)$. Take $y_i\in Y_i$
 such that $\pi(y_i)=x_i$, and let $Y'$ be the $\Cc$-submodule (or
 ${}^*\Cc$-submodule) of $Y$ generated by $\{y_1,\cdots,y_n\}$.
 Let $f':\ P\to X'$ be the corestriction of $f$. Since $X'$ and $Y'$ are finitely
 generated and $\pi_{|Y^{\prime}}$ is still an epimorphism, there exists $g:\ P\to X'$ such   
 that $f'=\pi\circ g$,
 and the projectivity of $P$ in $\Mm^\Cc$ follows from the commutativity of
 the diagram \equref{diagram}.
 \end{proof}

\section{Applications and examples}\selabel{5}
\subsection{Application to qF-rings}\selabel{5.1}
In \thref{3.1}, we gave equivalent conditions for the semiperfectness of a left locally projective
coring $\Cc$ over a right artinian ring $R$. In the case where $R$ is a qF-ring, more characterizations
are possible. This has been studied recently by El Kaoutit and G\'omex-Torrecillas 
(see \cite[Theorems 3.5, 3.8, 4.2]{Kaoutit3}. Using the results of the previous Sections, we find a different proof of these results.\\
First recall that a qF ring, or quasi-Frobenius ring, is a ring which is right artinian
and injective as a right $R$-module, or, equivalently, left artinian
and injective as a left $R$-module (in \cite{W}, these rings are called
noetherian QF rings). 
In this situation, $R$ is a cogenerator
of $\Mm_R$ and ${}_R\Mm$, see \cite[48.15]{W}. Since a qF-ring is a left and right
perfect ring, local projectivity is equivalent to projectivity. Also recall that flat modules
over qF-rings are projective. Let $R$ be a qF-ring, and assume that
$\Cc\in{}_R\Mm$ is flat (or, equivalently, (locally) projective). Then ${}^{\Cc}\Mm$
is a Grothendieck category, and the forgetful functor ${}^{\Cc}\Mm\to{}_{R}\Mm$
is exact and has a right adjoint $\Cc\ot_R-$. Since ${}_R\Mm$ has enough
injectives and the forgetful functor is exact, $\Cc\ot_R-$ preserves injectives. Now $R\in {}_R\Mm$ is injective
because $R$ is a qF-ring, so $\Cc=\Cc\ot_R R$ is an injective object of ${}^\Cc\Mm$,
and we can apply \prref{2.7}. We find that
$\Cc=\bigoplus_{i\in I} E(S_i)$, with $\bigoplus_{i\in I}S_i$ the decomposition 
of the left socle of $\Cc\in {}^\Cc\Mm$.\\
If $R$ is a qF-ring, then
the contravariant functors
$$(-)^{*}=\Hom_{R}(-,R):\ \Mm_R\to {}_R\Mm,~~
{}^{*}(-)={}_{R}\Hom(-,R):\ {}_R\Mm\to \Mm_R,$$
define an equivalence duality between the categories of finitely generated left 
$R$-modules and finitely generated right $R$-modules. More explicitely, 
every finitely generated left $R$-module $M$ is reflexive, that is, the map
$$\Phi_M:\ M\to ({}^{*}M)^{*},~~\Phi_{M}(m)(f)=f(m)$$
is an isomorphism. This result follows, for example, after we take $U=M=R$ in \cite[47.13(2)]{W}.\\
If $M$ is not finitely generated, then we still have the following result.

\begin{lemma}\lelabel{1.6}
Let $R$ be a qF ring and $M\in {}_R\Mm$-module.
Then $\im(\Phi_M)$ is dense in $({}^{*}M)^{*}$ with respect to the finite
topology on $\Hom_R({}^*M,R)$.
\end{lemma}

\begin{proof}
Take $T\in ({}^{*}M)^{*}$ and $F=\{f_{1},\dots,f_{n}\}\subset {}^{*}M$. We have to prove that there exists an $m\in M$ such that $T(f_{i})=\Phi_{M}(f_{i})=f_{i}(m)$. Let ${}^{\perp}F=\bigcap_{i=1}^n \Ker f_{i}\subset M$ and $N=M/ {}^{\perp}F$. Then we have a natural inclusion 
$$\frac{M}{\bigcap_{i=\overline{1,n}} \Ker\,f_{i}}\hookrightarrow \bigoplus_{i=1}^n\frac{M}{\Ker f_{i}}\simeq \bigoplus_{i=1}^n\im\,f_{i}\hookrightarrow R^{n} $$
and this shows that $N=M/{}^{\perp}F$ has finite length. Let $\pi:\ M\longrightarrow M/{}^{\perp}F=N$ be the canonical projection and consider its dual 
$\pi^{*}:\ {}^{*}N\longrightarrow {}^{*}M$. By the construction of $N$ as a factor module, there are left $R$-linear maps
$\overline{f_{i}}:\ N\longrightarrow R$ such that $\overline{f_{i}}\circ\pi= f_{i}$. 
Consider $t=T\circ\pi^{*}\in ({}^{*}N)^{*}$. As $N$ is finitely generated, $\Phi_{N}$ is an isomorphism (it gives the above stated duality between ${}_{R}{\mathcal M}$ and ${\mathcal M}_{R}$), so there is $n=\hat{m}=\pi(m)\in N$ such that $t=\Phi_{N}(n)$. Then $T(f_{i})=T(\overline{f_{i}}\circ\pi)=(T\circ\pi^{*})(\overline{f_{i}})=t(\overline{f_{i}}) =\Phi_{N}(n)(\overline{f_{i}})=\overline{f_{i}}(\pi(m))=f_{i}(m)$, 
as needed.
\end{proof}

If $\Cc$ is a left and right projective $R$-coring, then the duality is kept after we pass to the
categories of finitely generated $\Cc$-comodules: he functors 
${}^*(-)={}_R\Hom(-,R)$ and $(-)^*=\Hom_R(-,R)$
define an equivalence between the categories
${}^{{\rm fg}\Cc}\Mm$ and $\Mm^{{\rm fg}\Cc}$. To prove this, it suffices to show that $\Phi_M$
is left $\Cc$-colinear, or, equivalently, right $\Cc^*$-linear, for every finitely generated left
$\Cc$-comodule $M$, and this is a standard computation. From this duality and
\prref{3.2}, we obtain the following result.

\begin{corollary}\colabel{4.2}
Let $R$ be a qF-ring, and $\Cc$ an $R$-coring that is projective as a left and
right $R$-module. A finitely generated right $\Cc$-comodule $M$ is injective
(resp. projective) in $\Mm^\Cc$ if and only if $M^*$ is projective
(resp. injective) in ${}^{\Cc}\Mm$.
\end{corollary}

\begin{theorem}\thlabel{4.3}
Let $R$ be a qF-ring, and $\Cc$ an $R$-coring that is (locally) projective as a left and
right $R$-module.  The following assertions are equivalent.
\begin{enumerate}
\item[(i)] $\Rat^{\Cc}$ is exact;
\item[(ii)] $\Rat^{\Cc}({}^*\Cc)$ is dense in ${}^*\Cc$;
\item[(iii)] $\Rat^{\Cc}({}^*M)$ is dense in ${}^*M$ for every left
$\Cc$-comodule $M$;
\item[(iv)] $\Rat^{\Cc}({}^*Q)$ is dense in ${}^*Q$ for every left injective
$\Cc$-comodule $Q$;
\item[(v)] $\Rat^{\Cc}({}^*Q)$ is dense in ${}^*Q$ for every left injective
indecomposable $\Cc$-comodule $Q$;
\item[(vi)] ${}^*Q$ is ${}^*\Cc$-rational for every left injective
indecomposable $\Cc$-comodule $Q$;
\item[(vii)] $E(S)$ is finitely generated for every simple left comodule $S$;
\item[(viii)] every simple right $\Cc$-comodule has a finitely generated
projective cover;
\item[(ix)] $\Cc$ is right semiperfect.
\end{enumerate}
\end{theorem}

\begin{proof}
(i)$\Longleftrightarrow$(ii) follows from \prref{2.3}.\\
(ii)$\Longleftrightarrow$(v). As we have seen, $\Cc=\bigoplus_{i\in I} E(S_i)$,
and each injective indecomposable left $\Cc$-comodule is isomorphic to one
of the $E(S_i)$'s, because every comodule contains a simple comodule. The equivalence of (ii) and (v) then follows from \prref{2.4}.\\
(v)$\Longrightarrow$(iv). Every left injective comodule $Q$ is a direct sum
of injective indecomposable left $\Cc$-comodules (because its socle is essential), $Q=\bigoplus_{i\in I}Q_{i}$. Then
we have ${}^{*}Q=\prod_{i\in I}{}^{*}Q_{i}$ in $\Mm_{{}^{*}C}$ and $\bigoplus_{i\in I}\Rat^{\Cc}({}^{*}Q_{i})\subseteq\Rat^{\Cc}({}^*Q)\subset\prod_{i\in I}{}^{*}Q_{i}$ and then it all follows from \prref{1.3}.\\
(iv)$\Longrightarrow$(iii). Take $M\in {}^\Cc\Mm$ and an injective envelope
$f:\ M\to Q$ in ${}^\Cc\Mm$. We know that $\Rat^\Cc({}^*Q)$ is dense
in ${}^*Q={}_R\Hom(Q,R)$. \prref{1.5} then yields that
${}^*f(\Rat^\Cc({}^*Q))$ is dense in ${}_R\Hom(M,R)={}^*M$. But
${}^*f(\Rat^\Cc({}^*Q))\subset \Rat^\Cc({}^*M)$, so $\Rat^\Cc({}^*M)$ is dense
in ${}^*M$.\\
(iii)$\Longrightarrow$(iv)$\Longrightarrow$(v): trivial.\\
(i)$\Longleftrightarrow$(vii) follows from \prref{2.7}.\\
(vi)$\Longleftrightarrow$(vii) follows from \leref{2.5} and the fact that every injective 
indecomposable is isomorphic to one of the $E(S_{i})$'s.\\
(vii)$\Longleftrightarrow$(viii). Let $T$ be a simple right $\Cc$-comodule.
Then $T$ is finitely generated, and therefore a simple object in $\Mm^{{\rm fg}\Cc}$.
By the duality between ${}^{{\rm fg}\Cc}\Mm$ and $\Mm^{{\rm fg}\Cc}$, $T^*\in {}^{{\rm fg}\Cc}\Mm$ is simple, and $E(T^*)$ is
finitely generated by assumption. The monomorphism $T^*\to E(T^*)$ is essential,
so, using the duality, the dual map is a superfluous epimorphism
${}^*E(T^*)\to {}^{*}(T^*)\simeq T$. It follows from \coref{4.2} that ${}^*E(T^*)$
is projective, and, using again the duality, that it is finitely generated. Hence
${}^*E(T^*)$ is a finitely generated projective cover of $T$.
\end{proof}

A coring $\Cc$ is called left (resp. right) perfect if every object in ${}^\Cc\Mm$ (resp.
$\Mm^\Cc$) has a projective cover. We will now see that, over a qF-ring, perfectness on both
sides is equivalent to semiperfectness on both sides. First we need a Lemma.

\begin{lemma}\lelabel{4.4}
Let $R$ be a qF-ring, and $\Cc$ a right semiperfect coring that is both left and right profective over $R$. Then every
$0\neq M\in {}^\Cc\Mm$ contains a maximal subcomodule. Consequently the
Jacobson radical $J(M)$ is small in $M$.
\end{lemma}

\begin{proof}
${}^*M\in \Mm_{{}^*\Cc}$, and $\Rat^\Cc({}^*M)$ is dense in ${}^*M$,
by \thref{4.3}. Thus, if $\Rat^\Cc({}^*M)=0$, then ${}^*M=0$, which is
impossible since $R$ is a cogenerator in ${}_R\Mm$. So
$\Rat^\Cc({}^*M)\neq 0$, and we can take a nonzero simple right subcomodule
$S$ of $\Rat^\Cc({}^*M)$. Let $u:\ S\to \Rat^\Cc({}^*M)$ and
$v:\ \Rat^\Cc({}^*M)\to {}^*M$ be the inclusion maps. Then $u$ is
right $\Cc$-colinear, and $v$ is right ${}^*\Cc$-linear.
Now consider the composition $f=u^*\circ v^{*}\circ\phi$.
$$M\rTo^{\phi}({}^{*}M)^{*}  \rTo^{v^{*}} (\Rat^\Cc({}^*M))^{*}  \rTo^{u^{*}}  S^{*}.$$
A straightforward computation shows that
$v^{*}\circ\phi$ is left $\Cc^*$-linear, and therefore
$f=u^*\circ v^{*}\circ\phi$ is also left $\Cc^*$-linear. Now
$u^*\circ v^*$ is surjective, $\im \phi$ is dense in $({}^{*}M)^{*}$,
by \leref{1.6}, so $\im f= (u^*\circ v^{*})(\im \phi)$ is dense in $S^*$,
by \prref{1.5}. Since $S$ is simple, and therefore finitely generated,
the only dense submodule of $S^*$ is $S^*$ itself. So $f:\ M\to S^*$
is a surjective $\Cc^*$-linear morphism between the left $\Cc$-comodules
$M$ and $S^*$, hence it is a left $\Cc$-colinear surjection. Since $S^*$
is simple in ${}^*\Mm$, $\Ker f$ is a maximal subcomodule of $M$.
\end{proof}

\begin{proposition}\prlabel{4.5}
Let $R$ be a qF-ring, and $\Cc$ an $R$-coring which is left and right (locally)
projective over $R$. Then the following assertions are equivalent.
\begin{enumerate}
\item[(i)] $\Cc$ is left and right perfect;
\item[(ii)] $\Cc$ is left and right semiperfect.
\end{enumerate}
\end{proposition}

\begin{proof}
The implication (i)$\Rightarrow$(ii) is trivial. Conversely,
we will first show that $M/J(M)$ is a semisimple object in $\Mm^\Cc$, for any $M\in \Mm^\Cc$.
Take $\ol{x}\in M/J(M)$, and let $N$ be the subcomodule of $M/J(M)$ generated by
$\ol{x}$. Then $N\subset M/J(M)$, hence $J(N)\subset J(M/J(M))=0$. $N$ is finitely
generated, and therefore artinian. Let $N_1,\cdots,N_n$ be maximal subcomodules
of $N$ such that $\bigcap_{i=1}^n N_i=0$. Then $N=\bigoplus_{i=1}^n N/N_i$ is semisimple. This shows that every
$\ol{x}\in M/J(M)$ belongs to a semisimple subcomodule, so $M/J(M)$ is semisimple.\\
Since $\Cc$ is right semiperfect, there exists a projective cover $f:\ P\to M/J(M)$.
Since $P$ is projective, there exists $g\in \Mm^\Cc$ making the following diagram
commutative ($\pi$ is the canonical projection):
\begin{diagram}
& & P & & \\
& \SW^{g} & \dTo^{f} & & \\
M & \rTo^{\pi} & M/J(M) & \rTo & 0
\end{diagram}
Now $\pi(J(M)+g(P))=\pi(g(P))=f(P)=M/J(M)$, so $J(M)+g(P)=M$, since $\pi$ is surjective.
$\Cc$ is left semiperfect, hence, by \leref{4.4}, $J(M)\ll M$, and we conclude that
$g(P)=M$. So $g$ is surjective. $\Ker f\ll P$ and $\Ker g\subset \Ker f$, hence
$\Ker g\ll P$, and we conclude that $g:\ P\to M$ is a projective cover of $M$.
\end{proof}

\subsection{Examples}\selabel{5.2}
\begin{example}\exlabel{5.1}
Let $\Cc$ be a coring, and assume that $\Cc$ is finitely generated and projective as
a left $R$-module. Then $\Mm^\Cc$ is isomorphic to $\Mm_{{}^*\Cc}$, and $\Rat^\Cc$
is an isomorphism of categories. Hence $\Rat^\Cc$ is exact. $\Mm^\Cc$ has enough
projectives, but not necessarily projective covers. As an example, let $R$
be a non-semiperfect ring, and $\Cc=R$, the trivial $R$-coring. Then $\Mm^\Cc=\Mm_R$
is not semiperfect.
\end{example}

\begin{example}\exlabel{5.2}
Let $\Cc$ be a cosemisimple coring. Then $\Cc$ is left and right semiperfect, since
the categories of left and right $\Cc$-comodules are semisimple,
see \cite[19.14]{BW}, \cite{Kaoutit} and \cite{Kaoutit2}.
In this case, $\Cc$ is projective in ${}_R\Mm$ and $\Mm_R$, so $\Cc$ satisfies the
left and right $\alpha$-condition. $\Cc$ can then be written as a direct sum of finitely generated
left (or right) $\Cc$-comodules, and the functors $\Rat^\Cc$ and ${}^\Cc\Rat$ are exact.
So all the equivalent statements of \thref{4.3} hold, without the assumption that
the base ring $R$ is a qF-ring.
\end{example}

\begin{example}\exlabel{5.3}
To a ring morphism $\iota:\ R\to S$, we can associate the Sweedler coring $\Cc$.
As an $S$-bimodule, $\Cc=S\ot_R S$, and the comultiplication and counit are given by
the formulas
$$\Delta(s\ot_R s')=(s\ot_R 1)\ot_S (1\ot_R s')~~;~~\varepsilon(s\ot_R s')= ss'$$
The Sweedler coring is important in descent theory: the comodules over $\Cc$ are
exactly the descent data from \cite{KO} (in the commutative case) and \cite{Cipolla}
(in the noncommutative case). If $M\in \Mm^\Cc$, then $M$ descends to an $R$-module
$$M^{{\rm co}\Cc}=\{m\in M~|~\rho(m)=m\ot_R 1\}$$
For a detailed discussion, we refer to \cite{C}. It is
also easy to see that we have an isomorphism of $R$-algebras
$${}^*\Cc= {}_S\Hom(S\ot_R S,S)\cong {}_R\End(S)$$
(again, ${}_R\End(S)$ is a ring with the opposite composition as multiplication). Also notice that $S\subset {}_R\End(S)$ as algebras, by right multiplication.\\
If we assume that $S\in {}_R\Mm$ is locally projective, then $\Cc\in {}_S\Mm$ is
locally projective, and we can consider the functor
$$\Rat^\Cc:\ \Mm{}_{{}_R\End(S)}\to \Mm^\Cc$$
Let $M$ be a right ${}_R\End(S)$-module, and
take $m\in M$. Then $m\in \Rat^\Cc(M)$ if and only if there exists
$m_{[0]}\ot_R m_{[1]}\in M\ot_RS$ such that $m\cdot f=m_{[0]}f(m_{[1]})$, for
all $f\in {}_R\End(S)$. In particular,
\begin{eqnarray*}
\Rat^\Cc(M)^{{\rm co}\Cc}&=&\{m\in \Rat^\Cc(M)~|~\rho(m)=m\ot_R1\}\\
&=&\{m\in M~|~ m\cdot f=mf(1),~{\rm for~all}~f\in {}_R\End(S)\}
\end{eqnarray*}
$\Rat^\Cc(M)$ is a right $\Cc$-comodule, and therefore a right ${}_R\End(S)$-module,
and, by restriction of scalars, a right $S$-module. Therefore
\begin{equation}\eqlabel{5.3.1}
\Rat^\Cc(M)^{{\rm co}\Cc}\cdot S\subset \Rat^\Cc(M)
\end{equation}
If we take $M={}_R\End(S)$, then we see that
$$\Rat^{\Cc}(M)^{{\rm co}\Cc}=\{g\in{}_R\End(S)~|~(f\circ g)(s)=g(s)f(1),~
{\rm for~all}~f\in {}_R\End(S)\}$$
Take $h\in {}^*S={}_R\Hom(S,R)$. Then $\ol{h}=h\circ\iota\in {}_R\Hom(S,S)$,
and it follows easily that $\ol{h}\in \Rat^\Cc({}_R\End(S))^{{\rm co}\Cc}$.  
We will use this to show that $\Rat^\Cc(\End(S))$ is dense in
${}_R\End(S)$.\\
Take $f\in {}_R\End(S)$, and $F\subset S$ finite. Since $S$ is locally projective,
there are $h_1,\cdots,h_n\in {}^*S$ and $x_1,\cdots,x_n\in S$ such that
$$x=\sum_{k=1}^n h_k(x)x_k$$
for $x\in F$ and then a simple computation shows that
$$f(x)=\sum_{k=1}^n h_k(x)f(x_k)=\Bigl(\sum_{k=1}^n\ol{h}_k\cdot f(x_k)\Bigr)(x)$$
By \equref{5.3.1} and the the fact the above argument, $\sum_{k=1}^n\ol{h}_k\cdot f(x_k)\in \Rat^\Cc(\End(S))$.
So we have shown that $f$ coincides on $F$ to an element in $\Rat^\Cc(\End(S))$.
We conclude that $\Rat^\Cc({}^*\Cc)$ lies dense in ${}^*\Cc$, and, by \prref{2.3},
$\Rat^\Cc$ is exact.\\
If $S$ is pure as a left and right $R$-module, in particular if
 $S\in {}_R\Mm$ is faithfully flat, then the categories $\Mm_R$ and $\Mm^\Cc$
are equivalent (see \cite{C,Cipolla,KO}). In this case, $\Mm^\Cc$ has enough projectives.\\
If $S\in {}_R\Mm$ is faithfully flat and locally projective, then we have an explicit
description of $\Rat^\Cc(M)$, namely
$$\Rat^\Cc(M)=\Rat^\Cc(M)^{{\rm co}\Cc}\ot_RS,$$
with $\Rat^\Cc(M)^{{\rm co}\Cc}$ given by \equref{5.3.1}.
\end{example}


\begin{thebibliography}{99}
\bibitem{A}
J.Y. Abuhlail, {Rational modules for corings}, {\sl Comm. Algebra}
{\bf 31} (2003), 5793--5840.

\bibitem{AF}
F. Anderson and K. Fuller, ``Rings and categories of modules", {\sl Grad. Texts Math.}
{\bf 13}, Springer Verlag, Berlin, 1992.

\bibitem{B} 
T. Brzezi\'nski, The structure of corings.
      	Induction functors, Maschke-type theorem, and Frobenius and
      	Galois-type properties, {\sl Algebras and Representation
  	Theory} {\bf  5} (2002), 389--410.
 
 \bibitem{BW} 
T. Brzezi\'nski and R. Wisbauer, ``Corings and comodules", 
{\sl London Math. Soc. Lect. Notes Ser.} {\bf 309},
  Cambridge University Press, Cambridge, 2003.
  
  \bibitem{C}
S. Caenepeel, {\sl Galois corings from the descent theory point of view}, {\sl Fields
Inst. Comm} {\bf 43} (2004).

\bibitem{CVW}
S. Caenepeel, J. Vercruysse and Shuanhong Wang, 
Rationality properties for 
Morita contexts associated to corings,
in
``Hopf algebras in non-commutative
geometry and physics", S. Caenepeel and F. Van Oystaeyen, eds.,
{\sl Lecture Notes Pure Appl. Math.} {\bf 239},
Dekker, New York, 2004.

\bibitem{Cipolla}
M. Cipolla, {Discesa fedelmente piatta dei moduli}, {\sl Rendiconti
del Circolo Matematico di Palermo, Serie II} \textbf{25} (1976), 43--46. 

\bibitem{DascalescuNR}
S. D\v asc\v alescu, C. N\v ast\v asescu and \c S. Raianu,
``Hopf algebras: an Introduction'', {\sl Monographs Textbooks in Pure
Appl. Math.} {\bf 235}
Marcel Dekker, New York, 2001.

\bibitem{Kaoutit}
L. El Kaoutit and J. G\'omez Torrecillas, {Comatrix
corings: Galois corings, descent theory, and a structure
Theorem for cosemisimple corings}, {\sl Math. Z.}, {\bf 244} (2003), 887--906.

\bibitem{Kaoutit3}
L. El Kaoutit and J. G\'omez Torrecillas, {Morita duality for corings over quasi-Frobenius
rings}, in
``Hopf algebras in non-commutative
geometry and physics", S. Caenepeel and F. Van Oystaeyen, eds.,
{\sl Lecture Notes Pure Appl. Math.} {\bf 239},
Dekker, New York, 2004.

\bibitem{Kaoutit2}
L. El Kaoutit, J. G\'omez Torrecillas and F. J. Lobillo, Semisimple corings,
{\sl Algebra Coll.}, to appear.

\bibitem{KO}
M. Knus and M. Ojanguren, {``Th\'eorie de la Descente et Alg\`ebres
d'Azumaya"}, {\sl Lect. Notes in Math.} {\bf 389}, Springer Verlag,
Berlin, 1974.

\bibitem{NGT}
C. N\v ast\v asescu, J. G\'omez Torrecillas, Quasi-coFrobenius coalgebras,
{\sl J. Algebra} {\bf 174} (1995),  909--923.

\bibitem{Sweedler65}
M.E. Sweedler, {The predual Theorem to the Jacobson-Bourbaki
Theorem}, {\sl Trans. Amer. Math. Soc.} \textbf{213} (1975), 391--406.

\bibitem{W}
R. Wisbauer, ``Foundations of module and ring theory",
Gordon and Breach, Philadelphia, 1991





\end{thebibliography}
\end{document}